\documentclass[11pt,reqno]{amsart}
\usepackage{amssymb,latexsym}
\usepackage{graphicx}
\usepackage{mathtools}
\usepackage{color}
\usepackage[T1]{fontenc}
\usepackage{hyperref}
\usepackage{amsmath}
\usepackage{mathdots}
\usepackage{breqn}
\usepackage[toc,page]{appendix}
\usepackage{url}    
\usepackage{breqn}
\newcommand{\bburl}[1]{\textcolor{blue}{\url{#1}}}

\calclayout

\makeatletter
\newcommand{\monthyear}[1]{%
  \def\@monthyear{\uppercase{#1}}}
\newcommand{\volnumber}[1]{%
  \def\@volnumber{\uppercase{#1}}}
\AtBeginDocument{%
\def\ps@plain{\ps@empty
  \def\@oddfoot{\@monthyear \hfil \thepage}%
  \def\@evenfoot{\thepage \hfil \@volnumber}}
\def\ps@firstpage{\ps@plain}
\def\ps@headings{\ps@empty
  \def\@evenhead{%
    \setTrue{runhead}%
    \def\thanks{\protect\thanks@warning}%
    \uppercase{\ }\hfil}%
  \def\@oddhead{%
    \setTrue{runhead}%
    \def\thanks{\protect\thanks@warning}%
    \hfill\uppercase{Gaussian Behavior in Zeckendorf Decompositions From Lattices}}%
  \let\@mkboth\markboth
  \def\@evenfoot{%
    \thepage \hfil \@volnumber}%
  \def\@oddfoot{%
    \@monthyear \hfil \thepage}%
  }%
\footskip=25pt
\pagestyle{headings}%
}
\makeatother

\theoremstyle{plain}
\numberwithin{equation}{section}
\newtheorem{thm}{Theorem}[section]
\newtheorem{theorem}[thm]{Theorem}
\newtheorem{lemma}[thm]{Lemma}

\newcommand{\ncr}[2]{{#1 \choose #2}}

\newcommand{\ignore}[1]{}

\newcommand{\Var}{{\rm Var}}

\newcommand\be{\begin{eqnarray}}
\newcommand\ee{\end{eqnarray}}
\newcommand\bea{\begin{eqnarray}}
\newcommand\eea{\end{eqnarray}}
\newcommand\ben{\begin{enumerate}}
\newcommand\een{\end{enumerate}}

\newcommand{\N}{\mathbb{N}}

\begin{document}

\monthyear{June 2019}
\volnumber{Volume, Number}
\setcounter{page}{1}
\title{Central Limit Theorems for Compound Paths on the 2-Dimensional Lattice}

\author{Evan Fang, Jonathan Jenkins, Zack Lee, Daniel Li, Ethan Lu, Steven J. Miller, Dilhan Salgado, Joshua Siktar}


\address{\tiny{Department of Mathematical Sciences, Carnegie Mellon University, Pittsburgh, PA 15213}} \email{evanfang@andrew.cmu.edu}

\address{\tiny{Department of Mathematical Sciences, Carnegie Mellon University, Pittsburgh, PA 15213}} \email{jtjenkin@andrew.cmu.edu}

\address{\tiny{Department of Mathematical Sciences, Carnegie Mellon University, Pittsburgh, PA 15213}} \email{ztlee@andrew.cmu.edu}

\address{\tiny{Department of Mathematical Sciences, Carnegie Mellon University, Pittsburgh, PA 15213}} \email{drli@andrew.cmu.edu}

\address{\tiny{Department of Mathematical Sciences, Carnegie Mellon University, Pittsburgh, PA 15213}} \email{ethanlu@andrew.cmu.edu}

\address{\tiny{Department of Mathematical Sciences, Carnegie Mellon University, Pittsburgh, PA 15213, and Department of Mathematics and Statistics, Williams College, Williamstown, MA 01267}} \email{sjm1@williams.edu, Steven.Miller.MC.96@aya.yale.edu}

\address{\tiny{Department of Mathematical Sciences, Carnegie Mellon University, Pittsburgh, PA 15213}} \email{dsalgado@andrew.cmu.edu}

\address{\tiny{Department of Mathematical Sciences, Carnegie Mellon University, Pittsburgh, PA 15213}} \email{jsiktar@alumni.cmu.edu}




\date{\today}

\begin{abstract} Zeckendorf proved that every integer can be written uniquely as a sum of non-consecutive Fibonacci numbers $\{F_n\}$, and later researchers showed that the distribution of the number of summands needed for such decompositions of integers in $[F_n, F_{n+1})$ converges to a Gaussian as $n\to\infty$. Decomposition problems have been studied extensively for a variety of different sequences and notions of a legal decompositions; for the Fibonacci numbers, a legal decomposition is one for which each summand is used at most once and no two consecutive summands may be chosen. Recently, Chen et al. \cite{CCGJMSY} generalized earlier work to $d$-dimensional lattices of positive integers; there, a legal decomposition is a path such that every point chosen had each component strictly less than the component of the previous chosen point in the path. They were able to prove Gaussianity results despite the lack of uniqueness of the decompositions; however, their results should hold in the more general case where some components are identical. The strictly decreasing assumption was needed in that work to obtain simple, closed form combinatorial expressions, which could then be well approximated and led to the limiting behavior. In this work we remove that assumption through inclusion-exclusion arguments. These lead to more involved combinatorial sums; using generating functions and recurrence relations we obtain tractable forms in $2$ dimensions and prove Gaussianity again; a more involved analysis should work in higher dimensions.
\end{abstract}

\thanks{This work was supported by NSF Grant DMS1561945, Carnegie Mellon and Williams College.}

\maketitle

\tableofcontents
\section{Introduction}


Among the many fascinating properties of the Fibonacci numbers is the following observation, credited to Zeckendorf \cite{Ze}: Every positive integer admits a unique representation as a sum of non-adjacent Fibonacci numbers $\{F_n\}$, where\footnote{If we started with $F_0 = 0$ and $F_1 = 1$, then $F_2 = 1$ and we trivially lose uniqueness.} $F_1 = 1, F_2 = 2$ and $F_{n+1} = F_{n} + F_{n-1}$. Interestingly we can treat this property as an equivalent definition of the Fibonacci numbers: they are the only sequence from which every positive integer can be decomposed uniquely as a sum of non-adjacent terms. It turns out that there is often a relationship between rules for legal decompositions and sequences $\{G_n\}$, and the literature is now filled with many results on properties of the summands in legal decompositions of numbers in intervals $[G_n, G_{n+1})$ as $n\to\infty$. These range from the mean number of summands growing linearly, with the factor related to the roots of the characteristic polynomial of the recurrence, to the distribution of the number of summands converging to a Gaussian, to the distribution of gaps between summands; see for example \cite{Bes,Bow,Br,Day,Dem,FGNPT,Fr,GTNP,Ha, Ho, HW, Ke,Lek,Mw1,Mw2,Ste1,Ste2} and the references therein.


Most of the sequences studied have been one-dimensional. Additional sequences, such as those in \cite{CFHMN2,CFHMNPX}, appear two-dimensional but can be converted into one-dimensional sequences and attacked using existing techniques. This motivated Chen et. al. \cite{CCGJMSY} to consider a true multi-dimensional sequence by looking at paths among lattice points with non-negative integer coefficients. They defined a legal decomposition in $d$-dimensions to be a finite collection of lattice points for which
\begin{enumerate}
\item{each point is used at most once}, and
\item{if the point $(i_1, i_2, \dots, i_d)$ is included then all subsequent points $(i_1', i_2', \dots, i_d')$ have $i'_j <  i_j$ for all $j \in \{1, 2, \dots, d\}$ (i.e., \emph{all} coordinates must decrease between point in the decomposition and the next one).}
\end{enumerate}

They called the path of chosen lattice points a \textbf{simple jump path}; at each step, each component was \emph{strictly less than} the corresponding component of the previous step. One can construct a sequence on the lattice in many ways. For example, in two dimensions one can go along diagonal paths parallel to $y=-x$ and at each lattice point adding the first number which cannot be legally represented. The situation is slightly more involved in higher dimensions, though for most of the problems studied the values of the ordered points do not matter; what matters is the geometry of the lattice walks. In \eqref{ZeckendorfDiagonalSequenceSimp2D} we illustrate several diagonals' worth of entries when $d = 2$. Unlike for the Fibonacci sequence, we find that the uniqueness of these decompositions fails (for example, $25$ has two legal decompositions: $20+5$ and $24+1$).
\begin{eqnarray}
\begin{array}{cccccccccc}280 & \cdots & \cdots & \cdots & \cdots & \cdots & \cdots & \cdots & \cdots & \cdots \\157 & 263 & \cdots & \cdots & \cdots & \cdots & \cdots & \cdots & \cdots & \cdots \\84 & 155 & 259 & \cdots & \cdots & \cdots & \cdots & \cdots & \cdots & \cdots \\50 & 82 & 139 & 230 & \cdots & \cdots & \cdots & \cdots & \cdots & \cdots \\28 & 48 & 74 & 123 & 198 & \cdots & \cdots & \cdots & \cdots & \cdots \\14 & 24 & 40 & 66 & 107 & 184 & \cdots & \cdots & \cdots & \cdots \\7 & 12 & 20 & 33 & 59 & 100 & 171 & \cdots & \cdots & \cdots \\3 & 5 & 9 & 17 & 30 & 56 & 93 & 160 & \cdots & \cdots \\1 & 2 & 4 & 8 & 16 & 29 & 54 & 90 & 159 & \cdots \end{array}
\label{ZeckendorfDiagonalSequenceSimp2D}
\end{eqnarray}

These simple jump paths have a severe limitation as every coordinate must decrease. Thus in \eqref{ZeckendorfDiagonalSequenceSimp2D} we had to add 5 to our 2-dimensional sequence in the $(2,3)$ location, as we cannot use $4+1$ to get 5 as that is only horizontal movement. This strict decreasing condition was needed in \cite{CCGJMSY} to obtain simple closed form combinatorial expressions, which were then well approximated. Finally, this led to a proof that the limiting behavior of the distribution of the number of summands converges to a Gaussian.

This alternative, more general formulation leads to what we call a \(\textbf{generalized jump path}\). Formally, a generalized jump path from the lattice point \(p \in \mathbb{N}^d\) is a sequence of lattice points where each point is used only once, and if the point \((i_1, i_2, \dots, i_d)\) is in the sequence, then each subsequent point \((i_1', i_2', \dots, i_d')\) must satisfy the following properties:

\begin{itemize}

\item for all \(j\in \{1,2,\dots, d\}\), \(i_j \geq i_j'\), and

\item for at least one \(k \in \{1,2, \dots, d\}\), \(i_k > i_k'\).

\end{itemize}

These conditions imply that for each point in the sequence, at least one coordinate must decrease while the remaining coordinates cannot increase.
Below is the number of generalized jump paths in two dimensions from the point \((i,j)\) for \(i,j \leq 3\).
\begin{equation}
    \begin{matrix}\label{generalizedJumpPathsSmallPoints}
        \vdots & \vdots & \vdots & \vdots & \iddots \\
        8      & 40     & 152    & 504    & \cdots  \\
        4      & 16     & 52     & 152    & \cdots  \\
        2      & 6      & 16     & 40     & \cdots  \\
        1      & 2      & 4      & 8      & \cdots
    \end{matrix}
\end{equation}

We can find a recursive formula which allows us to calculate the total number of generalized jump paths starting at the point \(p=(p_1,p_2, \dots, p_d)\). For convenience we include the requirement that all paths end at one point outside the lattice, at the origin; note moving directly to the origin is equivalent to not choosing any additional points in a path and is thus considered a legal option. If \(S((p_1, p_2, \dots, p_d))\) represents the total number of paths starting at $p$, then we may do the following: partition all the generalized jump paths from \(p\) by the location of their first step. Either the path goes directly from \(p\) to the origin, or \(p\) first goes to some other lattice point \(a\). Note \(a\) must have at least one coordinate, say \(a_i\), such that \(a_i < p_i\), hence \(a \in \{[0,p_1] \times [0,p_2] \times \cdots \times [0, p_d]\}\setminus\{p\}\). By definition there are \(S(a)\) paths from the point $a$. Hence, summing over all possible \(a\) and the one additional case (immediately jumping to the origin, or equivalently not choosing any additional lattice points), we find a formula for \(S(p)\):
\begin{equation}\label{totalPathsRecurrence}
    S(p) \ = \  1 + \sum\limits_{\substack{a\neq p, \\ {a\in [0,p_1]\times [0, p_2]\times \cdots}}} S(a).
\end{equation}

We are able to perform an asymptotic analysis both on the number of generalized jump paths and the number of such paths of length $k$. Our main result is as follows.

\begin{thm}\label{thm:main} Let $X_{p, q}$ be the random variable denoting the number of generalized jump paths starting at a point $(p, q) \in \mathbb{N}^2$ and ending at the origin. Suppose $p := n$ and $q := cn$ for $n \in \N^+$ and $c \geq 1$ is fixed. Then $X_{p, q}$ converges to a Gaussian as $n \rightarrow \infty$, with mean $\frac{q + p + \sqrt{p^2 + 6pq + q^2}}{4}$ and variance $\frac{p+q}8+\frac{(p+q)^2}{8\sqrt{p^2+6pq+q^2}}$.
\end{thm}

In Section \ref{sec2}, we introduce some notation for our problem and prove some basic properties of unrestricted generalized jump paths. In Section \ref{sec3}, we use these properties in conjunction with various analytical and combinatorial methods to obtain a generating function for the number of paths to a fixed point as a function of path length. Using that result, we prove the Gaussianity in the limit of the number of summands in decompositions (Theorem \ref{thm:main}). We use a method similar to that found in \cite{CCGJMSY}; the difficulty in the argument is in determining a good count of the number of paths and, from that, a good estimate for the number of paths of a given length. We conclude with a brief discussion of future questions to study.

\section{Properties of Generalized Jump Paths} \label{sec2}

We define a \textbf{generalized jump path} of length $n$ that starts at $(p_1, p_2, \dots, p_d)$ to be a sequence of points $\{(x_{i,1},x_{i,2},\dots,x_{i,d})\}_{i=0}^n$ such that:
\begin{itemize}
    \item $(x_{0,1},x_{0,2},\dots,x_{0,d}) = (p_1, p_2, \dots, p_d)$,
    \item for all $i$ and $j$ we have  $x_{i,j} \geq x_{i+1,j}$,
    \item for all $i$ we have $(x_{i,1},x_{i,2},\dots,x_{i,d}) \neq (x_{i+1,1},x_{i+1,2},\dots,x_{i+1,d})$ \footnote{A natural future question would be to remove this condition, which would allow the same point to be used multiple times in a decomposition. See Section \ref{s:conclusion} for more information.}, and
    \item $(x_{n,1},x_{n,2},\dots,x_{n,d}) = (0, 0, \dots, 0)$.
\end{itemize}

Let $g((p_1, p_2, \dots, p_d),n)$ be the number of such paths of length $n$ starting at $(p_1, p_2, \dots, p_d)$. We define an \textbf{unrestricted generalized jump path} to be a sequence of points $\{(x_{i,1},x_{i,2},\dots,x_{i,d})\}_{i=0}^n$ such that:
\begin{itemize}
    \item $(x_{0,1},x_{0,2},\dots,x_{0,d}) = (p_1, p_2, \dots, p_d)$,
    \item for all $i$ and $j$ we have $x_{i,j} \geq x_{i+1,j}$, and
    \item for all $i$  we have $(x_{i,1},x_{i,2},\dots,x_{i,d}) \neq (x_{i+1,1},x_{i+1,2},\dots,x_{i+1,d})$.
\end{itemize}
Moreover, let $u((p_1,p_2,\dots,p_d),n)$ be the number of \emph{unrestricted} generalized jump paths starting from the point $(p_1, p_2, \dots, p_d);$ this is analogous to the definition of $g$.
Note that unrestricted generalized jump paths are simply generalized jump paths with the last restriction lifted (i.e., the sequence does not need to end at the bottom left corner).

We now establish and prove two basic properties for $g$ and $u$, the first of which was alluded to when we defined $u$.

\begin{lemma}[Unrestricted-Restricted Relationship] \label{unrestricted-restricted}
Let $v := (p_1,p_2,\dots,p_d)$. For all $n \in \mathbb{N}$,
    \be \label{unrestricted-restricted-recurrence-eq1}
        u(v,n) \ = \  g(v,n) + g(v,n+1).
    \ee
\end{lemma}

\begin{proof}
    The set of unrestricted jump paths of length $n$ that do not end at $(0,0,\dots,0)$ is bijective to the set of restricted jump paths of length $n+1$ that end at $(0,0,\dots,0)$. This immediately implies the result.
\end{proof}

\begin{theorem}[2-Dimensional Path Recurrence] \label{2d-recurrence}
    For all $p, q, n \in \mathbb{N}^+$,
    \begin{align} \label{2dimPathRecEqA}
        u((p,q),n) \ = \ ~ & u((p,q-1),n) + u((p,q-1),n-1) \nonumber        \\
                      & + u((p-1,q),n) + u((p-1,q),n-1)    \nonumber  \\
                      & - u((p-1,q-1),n) - u((p-1,q-1),n-1).
    \end{align}
\end{theorem}

\begin{proof}
    Let functions $\ell((p,q),n), d((p,q),n)$, and $m((p,q),n)$ denote the number of unrestricted jump paths starting from the point $(p,q)$ whose first jump is directly \emph{left}, directly \emph{down}, and both \emph{left and down}, respectively, in $n$ jumps. Then,
    \begin{align}\label{2dimPathRecEqAA}
        \ell((p,q),n)  & \ = \  \ell((p-1,q),n) + u((p-1,q),n-1)        \nonumber\\
        d((p,q),n) & \ = \  d((p,q-1),n) + u((p,q-1),n-1)           \nonumber\\
        m((p,q),n) & \ = \  d((p-1,q),n) + m((p-1,q),n)          \nonumber\\
                   & \ = \  \ell((p,q-1),n) + m((p,q-1),n)  \nonumber\\
                   & \ = \  u((p-1, q-1), n-1) + u((p-1, q-1), n).
    \end{align}

    Notice by the definitions of the functions $\ell((p, q), n), d((p, q), n)$, and $m((p, q), n)$ that
    \be \label{oneDirectionPartition}
    u((p,q),n) \ = \ \ell((p,q),n) + d((p,q),n) + m((p,q),n).
    \ee

    Using identity \eqref{oneDirectionPartition} gives
      \begin{align}
        u((p,q),n) \ = \ \ ~ & \ell((p-1,q),n) + u((p-1,q),n-1)        \nonumber      \\
                      & + d((p,q-1),n) + u((p,q-1),n-1)            \nonumber   \\
                      & + d((p-1,q),n) + m((p-1,q),n)         \nonumber  \\
                      & + \ell((p,q-1),n) + m((p,q-1),n) \nonumber          \\
                      & - (u((p-1, q-1), n-1) + u((p-1, q-1), n))   \nonumber \\
        \ = \  \ ~            & \ell((p-1,q),n) + d((p-1,q),n) + m((p-1,q),n)      \nonumber       \\
                      & + d((p,q-1),n) + \ell((p,q-1),n) + m((p,q-1),n)      \nonumber     \\
                      & + u((p-1,q),n-1) + u((p,q-1),n-1)                \nonumber      \\
                      & - (u((p-1, q-1), n-1) + u((p-1, q-1), n))       \nonumber       \\
        \ = \ \ ~            & u((p-1,q),n) + u((p,q-1),n)                 \nonumber           \\
                      & + u((p-1,q),n-1) + u((p,q-1),n-1)                  \nonumber  \\
                      & - u((p-1,q-1),n-1) - u((p-1,q-1),n), \label{2dimPathRecEqAA}
    \end{align}
    as desired.
\end{proof}


\section{2-Dimensional Generating Function} \label{sec3}
For every lattice point in $\mathbb{N}^2$ there is a 2-dimensional generating function for the lengths of the paths from that point. Explicitly, we denote $F_{p,q}(x)$ to be
\be \label{f-def}
  F_{p,q}(x) \ = \  \sum_{k=0}^{p+q} u((p,q),k) x^k.
\ee

Our main result in this section is an alternative form for $F_{p,q}$ that is more readily studied using asymptotic techniques.

\begin{theorem}\label{pathLenGenFn} If $p \leq q$, then
    \be
        F_{p,q}(x) \ = \  (1+x)^p\sum_{k=0}^{q} \binom{q}{k}\binom{p+k}{k} x^k.
    \ee
\end{theorem}

We give two proofs. The first is a pure generating function approach that only proves the claim for the case $p=q=n$, whereas the second is purely combinatorial and proves the theorem in generality. For technical convenience, one often analyzes all paths starting at a point on the main diagonal; thus if we restrict our investigation to this case, the simpler first proof suffices.

\subsection{Pure Generating Functions Method}
Consider the generating function
\begin{align}
    B(x,y,z) \ = \  \sum_{p,q,k \in \mathbb{Z}^+} u((p,q),k)x^py^qz^k.
\end{align}
Using Theorem \ref{2d-recurrence}, we immediately see that
\begin{align}
    B(x,y,z) \ = \  1+(1+z)(x+y-xy)B(x,y,z),
\end{align}
which implies that
\begin{align}
    B(x,y,z) \ = \  \dfrac{1}{1-(1+z)(x+y-xy)}.
\end{align}

Now we use a method adapted from \cite{Stan} to determine the central terms.   Define \be D(x,s,u)\ :=\ B(s,x/s,z),\ee and let $u := 1+z$; thus
\begin{align} \label{centralTermCompA}
    D(x,s,u) \ = \  \dfrac{1}{1-u(s+x/s-x)} \ = \  \dfrac{-s/u}{s^2-s(x+1/u)+x}.
\end{align}
Our generating function for the central terms will be the $s^0$ coefficient of $D(x,s,u)$.

We denote the two solutions for $s$ in the denominator of $D(x,s,u)$ as $\alpha$ and $\beta$, where for ease of reading we suppress the arguments of these functions as follows:
\be \alpha \ := \  \dfrac{ux+1-\sqrt{u^2 x^2 - 4 u^2 x + 2 u x + 1}}{2u},\ \ \ \beta \ := \  \dfrac{ux+1+\sqrt{u^2 x^2 - 4 u^2 x + 2 u x + 1}}{2u}. \ee
Using the method of partial fractions on \eqref{centralTermCompA}, we find that
\begin{align}  \label{centralTermCompB}
    D(x,s,u) \ = \  \frac{1}{u(\beta-\alpha)}\left(\frac{\alpha}{s-a}-\frac{\beta}{s-\beta}\right).
\end{align}
We expand each term by using the geometric series formula, which is applicable for sufficiently small values of the parameters, and get
\begin{align}  \label{centralTermCompC}
    D(x,s,u) \ = \  \frac{1}{u(\beta-\alpha)}\left(\sum_{n\geq1}\alpha^ns^{-n}+\sum_{n\geq0}\beta^{-n}s^n\right).
\end{align}
The $s^0$ term is clearly just $\frac{1}{u(\beta-\alpha)}$ as $\beta^{-0}s^0 = 1$.
Since $\beta-\alpha = \frac{\sqrt{u^2x^2-2u^2x+2ux+1}}{u}$, we have
\begin{align}  \label{centralTermCompD}
    B(x,x,z) \ = \  \frac{1}{\sqrt{u^2x^2-4u^2x+2ux+1}} \ = \  \frac{1}{\sqrt{(ux+1)^2-4u^2x}},
\end{align}
where we used that $u = z+1$.

We now represent $1 / \sqrt{(ux+1)^2-4u^2x}$ as a power series of the form $\sum_{i=0}b_i(z)x^i$. Differentiating, we obtain the recurrence relation
\begin{align} \label{centralTermCompE}
    b_i \ = \
    \begin{cases}
        1                                                         & i = 0   \\
        1+3z+2z^2                                                 & i = 1   \\
        \dfrac{(2i-1)(1+2z)(1+z)b_{i-1} - (1+z)^2(i-1)b_{i-2}}{i} & i \ge 2.
    \end{cases}
\end{align}
We define $a_i$ such that $a_i (1+z)^i = b_i$; then, this sequence satisfies the recurrence
\begin{align} \label{centralTermCompF}
    a_i \ = \
    \begin{cases}
        1                                             & i = 0   \\
        1+2z                                          & i = 1   \\
        \dfrac{(2i-1)(1+2z)a_{i-1} - (i-1)a_{i-2}}{i} & i \ge 2.
    \end{cases}
\end{align}
The solution to this recurrence relation is $a_i = P_i(2z+1)$, where $P_i$ is the $i$\textsuperscript{th} Legendre Polynomial
\begin{align}  \label{centralTermCompG}
    P_l(x) \ = \  \sum_{k=0}^{\ell} \binom{\ell}{k} \binom{-\ell-1}{k}\left(\frac{1-x}{2}\right)^k;
\end{align} see \cite{Ko}.
It follows that
\begin{align}  \label{centralTermCompH}
    a_n \ = \  P_n(2z+1) \ = \  \sum_{k=0}^n\binom{n}{k}\binom{n+k}{k}z^k.
\end{align}
Thus, changing variable names,
\begin{align}  \label{centralTermCompI}
    F_{n,n}(x) \ = \  b_n(x) \ = \  (1+x)^n a_n(x) \ = \  (1+x)^n\sum_{k=0}^n\binom{n}{k}\binom{n+k}{k}x^k,
\end{align}
which proves Theorem \ref{pathLenGenFn} in the $p=q=n$ case, as desired.

\subsection{Combinatorial Method}
We attempt to obtain a nicer formula for $g(p,n)$ by first relaxing our constraint to allow for paths with stationary points (where consecutive points are allowed to be exactly the same) and then later correcting for the over-counting. We do this as it is significantly easier to count the total number of paths with this relaxation in place.

Let $r(p,n)$ denote the number of paths from $p$ to the origin with this relaxed constraint.
Arguing as in the Stars and Bars problem\footnote{The number of ways to divide $N$ identical items into $G$ groups, where all that matters is how many items are placed in a group, is $\ncr{N+G-1}{G}$. To see this, consider $N+G-1$ items in a line; choosing $G-1$ of these partitions the remaining $N$ items into $G$ sets. The first set is all the elements up to the first chosen one, and so on. There is thus a one-to-one correspondence between the two counting problems. This method was first used for Zeckendorf decomposition problems in \cite{KKMW}, and has been successfully used in many works since then.}, it follows that
\be
    r(p,n) \ = \  \prod_{i = 1}^d\binom{p_i+n-1}{p_i}.
\ee
Let $s(p,n, k)$ be the number of paths to $p$ with at least $k$ stationary points.
Then
\be
    s(p,n,k) \ = \  \binom{n}k r(p,n-k).
\ee
By the principle of Inclusion-Exclusion, we have that the number of paths with no stationary points is
\begin{align} \label{orig-statement}
    g(p,n) \ = \  \sum\limits_{k=0}^{n-1} (-1)^k\binom{n}k r(p,n-k).
\end{align}

Now, the following identity and its proof serve as motivation for how to proceed in evaluating \eqref{orig-statement}.

\begin{lemma} \label{macc}
    For $m, n \in \mathbb{N}$,
    \be \label{maccEqn}
        \sum_{i = 0}^n \binom{n}{i} \binom{m+n-i}{k-i}(-1)^i \ = \  \binom{m}{k}.
    \ee
\end{lemma}

Recall that $[n]$ means $\{1, 2, \dots, n\}$. The justification is as follows.

\begin{proof}
    We view the inner term as counting the number of ordered pairs $(S,T)$
    such that $S \subseteq [n]$, $T \subseteq [m+n] \setminus S$, and $|S| + |T| = k$. Let the set of all such valid ordered pairs be $V$.
    We consider the sign-reversing involution $f :V \setminus E \mapsto V \setminus E$ (where $E$ is the set of ``exceptions,'' for which $f$ is ill-defined) defined by toggling the smallest term in $S \cup T$ between $S$ and $T$.
    For example, when $k = 5$,
    \be \label{involutionEx}
        f(\{2,3,5\}, \{7,8\}) \ = \  (\{3,5\}, \{2,7,8\}).
    \ee

    Given this definition, it's not difficult to see that $f$ is its own inverse (hence, an involution) and always "flips" the parity of $|S|$, sending ordered pairs with a positive coefficient in our sum to pairs with negative coefficients (and vice versa).
    Therefore, for all the ordered pairs on which $f$ is well-defined, our desired sum is 0.
    Thus, we have
    \be \label{sumToSetExclusion}
        \sum_{i = 0}^n \binom{n}{i} \binom{m+n-i}{k-i}(-1)^i \ = \  |E|,
    \ee   and it's not difficult to see that the only pairs $(S, T) \in E$ are those where $S = \emptyset$ and $T \subseteq [m+n] \setminus [n]$. There are clearly $\binom{m}{k}$ such choices for these sets, and the desired result follows. \end{proof}

Viewed properly, our formula in \eqref{orig-statement} looks similar to the above lemma; we re-write it as
\begin{align} \label{genPathsClosedA}
    g(p,n) \ = \  \sum_{i = 0}^n (-1)^i \binom{n}{i} \prod_{k=1}^d \binom{(p_k - 1) + n-i}{(n-1)-i}.
\end{align}
Furthermore, when $d = 1$, this formula agrees with \eqref{macc}, as expected.
Although similar methods may be applied in higher dimensions, we now consider the special case $d = 2$.
\begin{theorem}
    For $p, q \in \mathbb{N}$,
    \begin{align}  \label{genPathsClosedB}
        g((p,q),n) \ = \  \sum_{i=0}^{n-1} \binom{p-1}{i}\binom{p-1+n-i}{p}\binom{q}{n-i-1}.
    \end{align}
\end{theorem}
\begin{proof}
    We first assume without loss of generality that $p \leq q$.
    We have that
    \begin{align}  \label{genPathsClosedC}
        g(p,n) \ = \  \sum_{i = 0}^n (-1)^i \binom{n}{i} \binom{(p - 1) + n-i}{(n-1)-i} \binom{(q - 1) + n-i}{(n-1)-i}.
    \end{align}
    We now proceed in a similar manner to the proof of Lemma \ref{macc}.
    We view the inner term as counting the number of ordered pairs $(S,T, U)$
    such that $S \subseteq [n]$, $T \subseteq [p+n-1] \setminus S$, $U \subseteq [q+n-1] \setminus S$, and $|S| + |T| = |S| + |U| = n-1$ . Let the set of all such valid ordered pairs be $V$.
    We consider the sign-reversing involution $f :V \setminus E \mapsto V \setminus E$ (where $E$ is the set of ``exceptions'' for which $f$ is ill-defined) defined by toggling the smallest term in $S \cup (T\cap U)$ between all three sets. Given this definition, it's not difficult to see that $f$ is its own inverse (hence, an involution) and always ``flips'' the parity of $|S|$, sending ordered pairs with a positive coefficient in our sum to pairs with negative coefficients (and vice versa).
    Hence, for all the ordered pairs on which $f$ is well-defined, our desired sum is 0.
    Thus, we have
    \begin{align} \label{genPathsClosedD}
        \sum_{i = 0}^n (-1)^i \binom{n}{i} \binom{(p - 1) + n-i}{(n-1)-i} \binom{(q - 1) + n-i}{(n-1)-i}\ = \  |E|,
    \end{align}
    and our problem reduces to counting $|E|$.
    For $f$ to be ill-defined, it is necessary and sufficient for $S = \emptyset$ and $T\cap U \cap [n] = \emptyset$. \\

    In order to prove this, we begin by indexing every tuple $(S,T,U)$ by fixing $i := |T \cap U|$. With $i$ fixed, we then choose the $i$ terms that are common to both $T$ and $U$, which must be a subset of $[p + n - 1] \setminus [n]$. Consequently, there are exactly $\binom{p-1}{i}$ ways to do this.
    We may now freely choose the remaining $n-i-1$ members of $T$ (of which there are $\binom{p+n-i - 1}{n-i-1}$ ways to do so) and the remaining $n-i-1$ members of $U$ (for which there are $\binom{q}{n-1}$ ways to select). Multiplying these three terms (and simplifying the second term), we find exactly the desired term, concluding the proof.
\end{proof}

\begin{proof}[Proof (Theorem \ref{pathLenGenFn})]
    We now use \eqref{unrestricted-restricted-recurrence-eq1} to obtain
    \begin{align} \label{genPathsClosedE}
        u((p,q),n) \ = \  \sum_{i=0}^n \binom{p}{i} \binom{p+n-i}{p}\binom{q}{n-i}.
    \end{align}
    By the definition of $F_{p,q}$ given in \eqref{f-def} we know that
    \begin{align} \label{genPathsClosedF}
        F_{p,q}(x) & \ = \  \sum_{k=0}^{p+q}\sum_{i=0}^k \binom{p}{i} \binom{p+k-i}{p}\binom{q}{k-i}x^k          \nonumber \\
                   & \ = \ \Bigg(\sum_{i=0}^p\binom{p}{i}x^i\Bigg) \Bigg(\sum_{k=0}^q \binom{q}{k}\binom{p+k}{p}\Bigg)    \nonumber \\
                   & \ = \  (1+x)^p\sum_{k=0}^{q} \binom{q}{k}\binom{p+k}{k} x^k,
    \end{align}
    which proves Theorem \ref{pathLenGenFn}.
\end{proof}

\section{Proving 2-Dimensional Gaussianity} \label{sec4}

We begin this section with a little bit of notation. Let $X_{p,q}$ to be a random variable counting the length generalized jump path starting from the point $(p, q)$ chosen uniformly at random from all such paths.
From Theorem \ref{pathLenGenFn}, we have
\be \label{gaussDecomp}
    X_{p,q} \ = \  A_{p,q} + B_{p,q},
\ee
where $A$ and $B$ are independent random variables proportional to a binomial coefficient and a product of binomial coefficients; explicitly
\be \label{gaussDecompRandomVar}
    P(A = k) \ \propto\  \binom{p}{k}, \ \ \ \    P(B = k) \ \propto\  \binom{q}{k}\binom{p+k}{k}.
\ee
Note $A$ is just the well-studied binomial random variable, which converges to a Gaussian with mean $p/2$ and variance $p/4 $ as $p \to \infty$.

We now prove Theorem \ref{thm:main}. We start by restating it with the notation above: \emph{
Suppose $p := n$ and $q := cn$ for $n \in \N^+$ and $c \geq 1$ is fixed. Then $X_{p, q}$ converges to a Gaussian as $n \rightarrow \infty$, with mean $\frac{q + p + \sqrt{p^2 + 6pq + q^2}}{4}$ and variance $\frac{p+q}8+\frac{(p+q)^2}{8\sqrt{p^2+6pq+q^2}}$.}

Due to our definition of $X_{p, q}$ in \eqref{gaussDecomp}, it suffices to show $B$ also converges to a Gaussian, with mean $\frac{q-p+\sqrt{p^2+6pq+q^2}}{4}$ and variance $\frac{q-p}8+\frac{(p + q)^2}{8\sqrt{p^2+6pq+q^2}}$ for the aforementioned choices of $p$ and $q$. The proof will use similar techniques to those found in \cite{CCGJMSY}; the algebra is more involved due to the more complicated structure of the formulas for the number of paths.

To begin, let $p := n$ and $q := cn$, where $c \geq 1$ is fixed. Then $P(B=k)$ is proportional to a ratio of factorials:
\be \label{stirlingApprox1} P(B = k) \ \propto\  \frac{q!}{p!}\frac{(p+k)!}{k!k!(q-k)!} \ = \  \frac{q!}{p!}\frac{(n+k)!}{k!k!(cn-k)!}.\ee
Applying Stirling's  formula to approximate the factorials in \eqref{stirlingApprox1}, we find for $p, q$ large that
\be  \label{stirlingApprox2} P(B = k) \ \propto\  \frac{q!}{p!}\frac{(n+k)^{n+k+\frac12}}{2\pi k^{2k+1}(cn-k)^{cn-k+\frac12}}.\ee

We now define
\begin{align}
    M \ :=\ \frac{(n+k)^{n+k+\frac12}}{k^{2k+1}(cn-k)^{cn-k+\frac12}} \label{m-def}
\end{align}
(i.e., we remove the terms in \eqref{stirlingApprox2} that are constant with respect to $k$). Taking the logarithm of both sides of (\ref{m-def}), we find \be \log M\ =\ \Big(n+k+\frac12\Big)\log(n+k) - \Big(2k+1\Big) \log(k) - \Big(cn-k+\frac12\Big)\log(cn-k).\ee
Now we write $k$ as $an + t\sqrt{n}$, where $a := \frac{c-1+\sqrt{c^2+6c+1}}4$. This allows us to introduce a more natural variable for the number of steps in a path, where this number is written in terms of its distance (in natural units) from the mean. We Taylor expand, using \be \log(u+x)\ =\ \log(u)+\log\left(1 + \frac{x}{u}\right)\ =\ \log(u)+\frac{x}{u} - \frac12\left(\frac{x}{u}\right)^2+O\left(\frac{x^3}{u^3}\right). \ee We do this because later in Lemma \ref{eq:decay} we show that $k=an$ is the center of the distribution, and almost all (i.e., with probability 1 in the limit) the mass of the distribution is located where $t$ is small. We find
\begin{align} \label{gaussProofLogDeriv}
    \log M \ = \ ~ & \Big(n+k+\frac12\Big)\log\Big(n\Big(1+a+\frac{t}{\sqrt{n}}\Big)\Big)                                                               \nonumber\\
              & -\Big(2k+1\Big)\log\Big(n\Big(a+\frac{t}{\sqrt{n}}\Big)\Big)                                                                       \nonumber\\
              & -\Big(cn-k+\frac12\Big)\log\Big(n\Big(c-a-\frac{t}{\sqrt{n}}\Big)\Big)                                                             \nonumber\\
    \ = \ ~        & (n-cn-1)\log(n)                                                                                                                    \nonumber\\
              & + \Big(n+k+\frac12\Big)\log\Big(1+a+\frac{t}{\sqrt{n}}\Big)                                                                        \nonumber\\
              & -\Big(2k+1\Big)\log\Big(a+\frac{t}{\sqrt{n}}\Big)                                                                                  \nonumber\\
              & -\Big(cn-k+\frac12\Big)\log\Big(c-a-\frac{t}{\sqrt{n}}\Big)                                                                        \nonumber\\
    \ = \ ~        & (n-cn-1)\log(n)                                                                                                                    \nonumber\\
              & + \Big(n+k+\frac12\Big)\Big(\log(1+a)+\frac{t}{(1+a)\sqrt{n}}-\frac{t^2}{2(1+a)^2n}+O\Big(\frac{t^3}{n^{3/2}}\Big)\Big)            \nonumber\\
              & - \Big(2k+1\Big)\Big(\log(a)+\frac{t}{a\sqrt{n}}-\frac{t^2}{2a^2n}+O\left(\frac{t^3}{n^{3/2}}\right)\Big)                          \nonumber\\
              & - \left(cn-k+\frac12\right)\left(\log(c-a)-\frac{t}{(c-a)\sqrt{n}}-\frac{t^2}{2(c-a)^2n}+O\left(\frac{t^3}{n^{3/2}}\right)\right).
\end{align}

After standard but tedious computations, the details of which are in Appendix \ref{app:computations}, we obtain
\be \label{gaussProofE}
    \log M \ = \   -\chi t^2 + f(n) + O\left(\frac{t^3}{\sqrt{n}}\right),
\ee
where \bea \chi\ =\ \frac{2 c^2 + 10 c^3 - 10 c^4 - 2 c^5 +
        2 c^2 \sqrt{1 + c (6 + c)} + 4 c^3 \sqrt{1 + c (6 + c)} +
        2 c^4 \sqrt{1 + c (6 + c)}}{8c^4}\nonumber\\ \eea and $f(n)$ is independent of $k$ and $t$.

We now prove the following lemma, which demonstrates that the tails of this distribution decay sufficiently quickly.

\vspace{0.2 cm}

\begin{lemma} \label{eq:decay}
    As $n \to \infty$, $P(|t|>n^{0.1}) \to 0$.
\end{lemma}

\begin{proof}
    First, we show $P(t>n^{0.1})$ goes to 0; the proof for $P(t<-n^{0.1})$ follows similarly.

    Note that if we increase $k$ by 1, then the number of paths of length $k$ increases by a factor of around $\frac{(p+k)(q-k)}{k^2}$ (note that this is strictly decreasing in $k$). Plugging in $k = an+t\sqrt{n}$, we see that the proportion at which the number of paths increases is
    \begin{align} \label{pathIncrProportion}
        1 - \left(\frac{8\sqrt{c^2+6c+1}}{(1+c)^2+(c-1)\sqrt{c^2+6c+1}}\frac{t}{\sqrt{n}}\right) + O\left(\frac{t^2}n\right).
    \end{align}
    Let the coefficient of the $t/\sqrt{n}$ term be $c'$. Plugging in $t=1$, and ignoring the higher order terms means that for $k \geq an+\sqrt{n}$, the ratio will be at most $r = 1- c'\sqrt{n}$. Let $B_0$ be the number of paths of length $k$ when $t = 1$. The number of paths with $t>n^{0.1}$ is bounded above by a geometric series of first term $B_0\cdot r^{n^0.6-n^0.5}$ and ratio $r$, and we have
    \begin{align}
        r^{n^{0.5}} \ = \  \left(1-\frac{c'}{\sqrt{n}}\right)^{\sqrt{n}} \approx e^{-c'}<1.
    \end{align}
    Thus as $n$ goes to infinity, $B_0\cdot r^{n^0.6-n^0.5} = B_0\cdot \left(e^{-c'}\right)^{n^{0.1}-1}$, where the last exponent, $n^{0.1}-1$, goes to infinity.
    Now note that the total number of paths of any length is strictly greater than $B_0$, so the probability that $t > n^{0.1}$ is at most
    \begin{align}
        \frac{B_0\cdot \left(e^{-c'}\right)^{n^{0.1}-1}/(1-r)}{B_0} \ = \  \frac{\sqrt{n}}{c'}\left(e^{-c'}\right)^{n^{0.1}-1},
    \end{align}
    which goes to 0 as $n$ grows. Thus, the probability that $t > n^{0.1}$ goes to 0 as $n$ gets large, as desired.
\end{proof}

Consequently, we know that $t^3/\sqrt{n} \to 0$ as $n \to \infty$, so by sending $p$ and $q$ to $\infty$,
\begin{align} \label{proportionEquality}
    P(B = k) \ = \  Ce^{-\chi t^2} \ = \  Ce^{-\frac{(an-k)^2}{n/\chi}},
\end{align}
where $C$ is some constant in terms of $p$ and $q$. This is the equation for a Gaussian Distribution with mean
\begin{align} \label{GaussianBMean}
    an \ = \  \frac{q-p+\sqrt{p^2+6pq+q^2}}4
\end{align}
and variance
\begin{align} \label{GaussianBVariance}
    \frac{n}{2\chi} \ = \  n\left(\frac{c-1}8+\frac{c^2+2c+1}{8\sqrt{1+6c+c^2}}\right) \ = \  \frac{q-p}8+\frac{q^2+2pq+p^2}{8\sqrt{p^2+6pq+q^2}}.
\end{align}
We conclude that $X_{p,q}$ is Gaussian with mean
\begin{align}
    E[X_{p,q}] \ = \  E[A_{p,q}] + E[B_{p,q}] \ = \  \frac{p}2+\frac{q-p+\sqrt{p^2+6pq+q^2}}4 \ = \  \frac{p+q}4+\frac{\sqrt{p^2+6pq+q^2}}4,
\end{align} and variance
\begin{eqnarray}
    \Var\left(X_{p,q}\right) & \ = \ &  \Var\left(A_{p,q}\right) + \Var\left(B_{p,q}\right) \nonumber\\ & \ = \  & \frac{p}4+\frac{q-p}8+\frac{q^2+2pq+p^2}{8\sqrt{p^2+6pq+q^2}} \ = \  \frac{p+q}8+\frac{(p+q)^2}{8\sqrt{p^2+6pq+q^2}}.
\end{eqnarray}
This completes the proof of Theorem \ref{thm:main}. \hfill $\Box$

\section{Further Work and Open Questions} \label{sec5}
\label{s:conclusion}

We conclude with some suggestions of future research.

\begin{enumerate}
    \item Is it possible to generalize the generating function approach for points not on the diagonal?
          Although empirical results suggest similar behavior, the associated Taylor expansion becomes significantly harder to work with.

    \item How do our results generalize to higher dimensions? Our combinatorial approach admits a natural extension to higher dimensions, although the casework becomes significantly more cumbersome. Furthermore, it is of note that the generating function approach does not generalize to three or more dimensions.

    \item How quickly does our distribution converge to a Gaussian?

    \item What is the behavior if we allow a point to be used more than once? Of course, if we can use a point arbitrarily many times it is unclear how to define the terms of our sequence; thus the natural question would be what happens if each point can be used at most $T$ times, for some fixed $T$.

    \item There are even more combinatorial questions worth exploring. Many different additional restrictions can be put on the compound jump paths. One such restriction is prohibiting any path from visiting any points that lie above the line $y = x$.
\end{enumerate}

\appendix


\section{Gaussianity Calculations}\label{app:computations}

We give the simplification of $\log M$ from \eqref{m-def}. We have that
\begin{align}
    \log M & \ = \
    -\frac{t}{a \sqrt{n}}+\frac{t}{2 (1+a) \sqrt{n}}+\frac{t}{2 (-a+c) \sqrt{n}}-\frac{2 k t}{a \sqrt{n}}+\frac{k t}{(1+a) \sqrt{n}}                     \nonumber\\
           & \hspace{0.5cm} -\frac{k t}{(-a+c) \sqrt{n}}+\frac{\sqrt{n} t}{1+a}+\frac{c \sqrt{n} t}{-a+c}-\frac{t^2}{2 (1+a)^2}+\frac{c t^2}{2 (-a+c)^2} \nonumber\\
           & \hspace{0.5cm} +\frac{t^2}{2 a^2 n}-\frac{t^2}{4 (1+a)^2 n}+\frac{t^2}{4 (-a+c)^2 n}+\frac{k t^2}{a^2 n}-\frac{k t^2}{2 (1+a)^2 n}          \nonumber\\
           & \hspace{0.5cm}-\frac{k t^2}{2 (-a+c)^2 n}-\log(a)-2 k \log(a)+\frac{1}{2}\log(1 + a)+k \log(1 + a)                                          \nonumber\\
           & \hspace{0.5cm}+ n \log(1 + a)-\frac{1}{2} \log(-a+c)+k \log(-a+c)-c n \log(-a+c).
\end{align}

Substituting $a = \frac{c-1+\sqrt{c^2+6c+1}}4$, we find that this expression for $\log M$ is equal to
\begin{align}
     & -\frac{4 t}{\left(-1+c+\sqrt{1+6 c+c^2}\right) \sqrt{n}}+\frac{t}{2 \left(c+\frac{1}{4} \left(1-c-\sqrt{1+6 c+c^2}\right)\right) \sqrt{n}}                                                                    \nonumber\\
     & +\frac{t}{2 \left(1+\frac{1}{4} \left(-1+c+\sqrt{1+6 c+c^2}\right)\right) \sqrt{n}}+\frac{c \sqrt{n} t}{c+\frac{1}{4} \left(1-c-\sqrt{1+6 c+c^2}\right)}                                                      \nonumber\\
     & +\frac{\sqrt{n}t}{1+\frac{1}{4} \left(-1+c+\sqrt{1+6 c+c^2}\right)}+\frac{c t^2}{2 \left(c+\frac{1}{4} \left(1-c-\sqrt{1+6 c+c^2}\right)\right)^2}                                                            \nonumber\\
     & -\frac{t^2}{2\left(1+\frac{1}{4} \left(-1+c+\sqrt{1+6 c+c^2}\right)\right)^2}+\frac{8 t^2}{\left(-1+c+\sqrt{1+6 c+c^2}\right)^2 n}                                                                            \nonumber\\
     & +\frac{t^2}{4 \left(c+\frac{1}{4}\left(1-c-\sqrt{1+6 c+c^2}\right)\right)^2 n}-\frac{t^2}{4 \left(1+\frac{1}{4} \left(-1+c+\sqrt{1+6 c+c^2}\right)\right)^2 n}                                                \nonumber\\
     & -\frac{8 t \left(\frac{1}{4} \left(-1+c+\sqrt{1+6 c+c^2}\right) n+\sqrt{n} t\right)}{\left(-1+c+\sqrt{1+6 c+c^2}\right) \sqrt{n}}-\frac{t \left(\frac{1}{4} \left(-1+c+\sqrt{1+6
            c+c^2}\right) n+\sqrt{n} t\right)}{\left(c+\frac{1}{4} \left(1-c-\sqrt{1+6 c+c^2}\right)\right) \sqrt{n}}                                                                                                \nonumber\\
     & +\frac{t \left(\frac{1}{4} \left(-1+c+\sqrt{1+6c+c^2}\right) n+\sqrt{n} t\right)}{\left(1+\frac{1}{4} \left(-1+c+\sqrt{1+6 c+c^2}\right)\right) \sqrt{n}}+\frac{16 t^2 \left(\frac{1}{4} \left(-1+c+\sqrt{1+6
            c+c^2}\right) n+\sqrt{n} t\right)}{\left(-1+c+\sqrt{1+6 c+c^2}\right)^2 n}.
\end{align}

Continuing the simplification, we see the above equals
\begin{align}                                                                                                                                                  \nonumber\\
     & -\frac{t^2 \left(\frac{1}{4} \left(-1+c+\sqrt{1+6 c+c^2}\right) n+\sqrt{n}t\right)}{2 \left(c+\frac{1}{4} \left(1-c-\sqrt{1+6 c+c^2}\right)\right)^2 n}-\frac{t^2 \left(\frac{1}{4} \left(-1+c+\sqrt{1+6 c+c^2}\right) n+\sqrt{n}
        t\right)}{2 \left(1+\frac{1}{4} \left(-1+c+\sqrt{1+6 c+c^2}\right)\right)^2 n}                                                                                                                                                   \nonumber\\
     & -\log\left(\frac{1}{4} \left(-1+c+\sqrt{1+6 c+c^2}\right)\right)
    \nonumber\\
     & -2
    \left(\frac{1}{4} \left(-1+c+\sqrt{1+6 c+c^2}\right) n+\sqrt{n} t\right) \log\left(\frac{1}{4} \left(-1+c+\sqrt{1+6 c+c^2}\right)\right)
    \nonumber\\
     & -\frac{1}{2}
    \log\left(c+\frac{1}{4} \left(1-c-\sqrt{1+6 c+c^2}\right)\right)-c n \log\left(c+\frac{1}{4} \left(1-c-\sqrt{1+6 c+c^2}\right)\right)                                                                                                \nonumber\\
     & +\left(\frac{1}{4}
    \left(-1+c+\sqrt{1+6 c+c^2}\right) n+\sqrt{n} t\right) \log\left(c+\frac{1}{4} \left(1-c-\sqrt{1+6 c+c^2}\right)\right)                                                                                                              \nonumber\\
     & +\frac{1}{2} \log\left(1+\frac{1}{4}
    \left(-1+c+\sqrt{1+6 c+c^2}\right)\right)+n \log\left(1+\frac{1}{4} \left(-1+c+\sqrt{1+6 c+c^2}\right)\right)                                                                                                                        \nonumber\\
     & +\left(\frac{1}{4} \left(-1+c+\sqrt{1+6
            c+c^2}\right) n+\sqrt{n} t\right) \log\left(1+\frac{1}{4} \left(-1+c+\sqrt{1+6 c+c^2}\right)\right).
\end{align}

Simplifying and collecting like terms, we have $\log M$ equals
\begin{align}\frac{1}{8 c^4 n}         & \left(t^2+6 c t^2+7 c^2 t^2+4 c^3 t^2-3 c^4 t^2-6 c^5 t^2-c^6 t^2 +\sqrt{1+c (6+c)} t^2\right.                 \nonumber\\
                              & +3 c \sqrt{1+c (6+c)} t^2+2 c^2 \sqrt{1+c (6+c)} t^2                                                           \nonumber\\
                              & \left. -2 c^3 \sqrt{1+c (6+c)} t^2+3 c^4 \sqrt{1+c (6+c)} t^2+c^5 \sqrt{1+c (6+c)} t^2\right)                  \nonumber\\
    +\frac{1}{8 c^4 \sqrt{n}} & \left(2 c^2 t+2 c^3 t+10 c^4
    t+2 c^5 t+2 c^2 \sqrt{1+c (6+c)} t-4 c^3 \sqrt{1+c (6+c)} t \right.                                                                        \nonumber\\
                              & -2 c^4 \sqrt{1+c (6+c)} t-2 t^3-12 c t^3-6 c^2 t^3+8 c^3 t^3-6 c^4 t^3-12 c^5 t^3                              \nonumber\\
                              & -2 c^6   t^3-2 \sqrt{1+c (6+c)} t^3-6 c \sqrt{1+c (6+c)} t^3+4 c^2 \sqrt{1+c (6+c)} t^3                        \nonumber\\
                              & \left.-4 c^3 \sqrt{1+c (6+c)} t^3+6 c^4 \sqrt{1+c (6+c)} t^3+2 c^5 \sqrt{1+c (6+c)} t^3\right)                 \nonumber\\
    + \frac{1}{8 c^4}         & \left(-2 c^2 t^2-10 c^3 t^2+10 c^4 t^2+2 c^5 t^2-2 c^2 \sqrt{1+c (6+c)} t^2-4 c^3 \sqrt{1+c (6+c)} t^2 \right. \nonumber\\
                              & -2 c^4 \sqrt{1+c (6+c)} t^2+8 c^4 \log(8)-4 c^4 \log\left(4+12 c-4 \sqrt{1+c (6+c)}\right)                     \nonumber\\
                              & \left. -8 c^4 \log\left(-1+c+\sqrt{1+c (6+c)}\right)+4 c^4 \log\left(3+c+\sqrt{1+c
                (6+c)}\right)\right)                                                                                                           \nonumber\\
    +\frac{\sqrt{n}}{8c^4}
                              & \left(8 c^4 t \log\left(1+3 c-\sqrt{1+c (6+c)}\right)
    -16 c^4 t \log\left(-1+c+\sqrt{1+c (6+c)}\right)\right.                                                                                    \nonumber\\
                              & \left.+8c^4 t \log\left(3+c+\sqrt{1+c (6+c)}\right)\right).
\end{align}

Finally, we obtain that $\log M$ is
\begin{align}
    +\frac{n}{8 c^4} & \left(-8 c^4 \log(4)+8 c^5 \log(4)-2 c^4 \sqrt{1+c
            (6+c)} \log(4)-2 c^4 \log\left(1+3 c-\sqrt{1+c (6+c)}\right) \right.                                                 \nonumber\\
                     & -6 c^5 \log\left(1+3 c-\sqrt{1+c (6+c)}\right)+2 c^4 \sqrt{1+c (6+c)}
    \log\left(1+3 c-\sqrt{1+c (6+c)}\right)                                                                                      \nonumber\\
                     & +4 c^4 \log\left(-1+c+\sqrt{1+c (6+c)}\right)-4 c^5 \log\left(-1+c+\sqrt{1+c (6+c)}\right)                \nonumber\\
                     & -4c^4 \sqrt{1+c (6+c)} \log\left(-1+c+\sqrt{1+c (6+c)}\right)+6 c^4 \log\left(3+c+\sqrt{1+c (6+c)}\right) \nonumber\\
                     & \left.+2 c^5 \log\left(3+c+\sqrt{1+c
                (6+c)}\right)+2 c^4 \sqrt{1+c (6+c)} \log\left(4 \left(3+c+\sqrt{1+c (6+c)}\right)\right)\right).
\end{align}

We now verify that the $t \sqrt{n}$ term has coefficient 0.
Our desired coefficient is
\bea & &  8 c^4 t \log\left(1+3 c-\sqrt{1+c (6+c)}\right)-16 c^4 t \log\left(-1+c+\sqrt{1+c (6+c)}\right)\nonumber\\ & & \ \ \ \ +\ 8c^4 t \log\left(3+c+\sqrt{1+c (6+c)}\right).
\eea

Exponentiating this expression, we obtain
\be
    \frac{\left(1+3 c-\sqrt{1+c (6+c)}\right) \left(3+c+\sqrt{1+c (6+c)}\right)}{\left(-1+c+\sqrt{1+c (6+c)}\right)^2} \ = \  1
\ee
as desired, completing the computations.

Consequently, as claimed, we have that $\log M$ is of the form
\be \label{logFactorForm}
    \log M \ = \   -\chi t^2 + f(n) + O\left(\frac{t^3}{\sqrt{n}}\right),
\ee
where $\chi$ is defined as \be\chi \ = \  \frac{2 c^2 + 10 c^3 - 10 c^4 - 2 c^5 +
        (2 c^2 + 4c^3 + 2c^4)\sqrt{1 + c (6 + c)}}{8c^4},\ee and $f(n)$ is independent of $k$ or $t$.



\newcommand{\etalchar}[1]{$^{#1}$}

\ \\

\end{document}